\documentclass[12pt]{article}
\usepackage{amsmath}
\usepackage{amsfonts}
\usepackage{amssymb}
\usepackage{amsthm}
\usepackage{stackrel}

\usepackage{color}

\usepackage{figsize}

\usepackage{graphics}

\newtheorem{thm}{Theorem}[section]

\newcommand{\go}[1]{\mathfrak{#1}}
\def\binom#1#2{{#1}\choose{#2}}
\newcommand{\R}{{\rm I}\kern-0.18em{\rm R}}
\newcommand{\1}{{\rm 1}\kern-0.25em{\rm I}}
\newcommand{\E}{{\rm I}\kern-0.18em{\rm E}}
\newcommand{\p}{{\rm I}\kern-0.18em{\rm P}}

\makeatletter

\def\@fnsymbol#1{\ensuremath{\ifcase#1\or a\or b\or c\or d\or \e\or f\or *\dagger 	\or \ddagger\ddagger \else\@ctrerr\fi}}
\title{Characterizations of Two-Points and Other Related Distributions}
\author{L.B. Klebanov\footnote{Department of Probability and Mathematical Statistics, Charles University, Prague, Czech Republic. e-mail: lev.klebanov@mff.cuni.cz}}
\date{}
\begin{document}
\maketitle

\begin{abstract}
We provide new characterizations of two-points and some related distributions. We use properties of independence and/or identity of the  distributions of suitable linear forms of random variables.

 
\end{abstract}

\begin{section}{Characterization by the identical distribution property}\label{sec1}
\setcounter{equation}{0} 

In a recent paper \cite{KH} we have given characterizations of the hyperbolic secant distribution based on properties of identical distribution of linear forms with random coefficients and on the property of independence of such forms. These characterization properties are,  
in some sense, `similar' (or analogous) to characterization properties of the Gaussian distribution. Here we 
use the same technique, the method of intensive monotone operators, to provide further characterizations results for two-points distributions 
and for other related distributions. 

\begin{thm}\label{th1} 
Let $X_1, X_2$ be two independent copies of a symmetric random variable $X$. Suppose  $\varepsilon$ is a random variable taking two values, $0$ and $1,$ each with probability  $\tfrac12$ and such that $X_1, X_2, \varepsilon$ are independent.
Then the relation 
\begin{equation}\label{eq1}
\varepsilon X_1 \ \stackrel{d}{=} \ \tfrac12(X_1+X_2) 
\end{equation}
holds if and only if $X$ takes only two values, $-a$ and $a$,  for some $a > 0$, each with probability $\tfrac12.$ 
\end{thm}
\begin{proof}
Let $g(t)$ be characteristic function of the random variable $X$. The equation (\ref{eq1}) can be written in the form
\begin{equation}\label{eq2}
\frac{1}{2}\bigl(g(t)+1\bigr) =g^2(t/2).
\end{equation}
It is easy to verify that for any $a \geq 0$ the function $\cos(at)$ is a solution of (\ref{eq2}). We have to prove that there are no other solutions of this equation. For that aim we use the method of intensively monotone operators (see \cite{KKM}). Define now the following operator
\[ (\mathbf{A}g)(t)=2g^2(t/2)-1. \]
Denote by $\mathcal{E}_+$ the set of all non-negative continuous functions in the interval $[0,T]$ for some fixed number $T>0,$ 
 such that $g(t)$ has no zeros in $[-T,T]$. By using the terminology and arguments from \cite{KKM}, 
 it is easy to check that the operator $\mathbf{A}$ is intensively monotone from $\mathcal{E}_+$ to 
 the space of all continuous functions in $[0,T].$

Consider the family $\{ \varphi(a t),\; a>0\}$, where $\varphi (t) = \cos(t)$. According to Example 1.3.1 from \cite{KKM} this family is strongly $\mathcal{E}_+$-positive.
The statement of Theorem \ref{th1} follows now from  Theorem 1.1.1 in \cite{KKM}. 
\end{proof}

Let us give now a more general result.

\begin{thm}\label{th1a}
Let $X$ be a random variable symmetric with respect to the origin $0$ and $n \geq 2$ be a natural number. Consider $n$ independent 
copies of $X$, namely $X_1, X_2, \dots ,X_n.$ Let $\varepsilon_n$ be a discrete random variable such that $X_1,X_2,\ldots ,X_n,\varepsilon_n$ are independent. If $n$ is odd, we suppose that $\varepsilon_n$ takes values $0,2,\ldots ,n-1$ with the following probabilities
\[ {\bf P}[\varepsilon_n = n-2k] =\frac{1}{2^{n-1}} {\binom{n}{k}},\quad k=0, \ldots ,(n-1)/2.  \]
In the case when $n$ is even $\varepsilon_n$ takes values $0,2,\ldots ,n-2$ with probabilities
\[ {\bf P}[\varepsilon_n = n] = \frac{1}{2^{n}} {\binom{n}{n/2}}+\frac{1}{2^{n-1}},\]
\[ {\bf P}[\varepsilon_n=n-2k]  =\frac{1}{2^{n-1}} {\binom{n}{k}},\quad k=1, \ldots ,n/2-1 .\]
Then the relation
\begin{equation}\label{eq2a}
\varepsilon_n X_1 \ \stackrel{d}{=} \ X_1+X_2+\ldots +X_n 
\end{equation}
holds if and only if $X$ takes only two values, $-a$ and $a$ for some $a > 0$, each with probability $\tfrac12$.  
\end{thm}
\begin{proof} The relation (1.3) can be written equivalently in terms of the characteristic function $f$ of $X_1$: 
\begin{equation}\label{eq2b}
f^{n}(t)=\frac{1}{2^{n-1}} \sum_{k=0}^{(n-1)/2} {\binom{n}{k}}f((n-2k)t) \ \mbox{ for odd } \ n, 
\end{equation} 
\begin{equation}\label{eq2c}
f^{n}(t)=\frac{1}{2^{n-1}} \sum_{k=0}^{n/2-1} {\binom{n}{k}}f((n-2k)t)+\frac{1}{2^{n}}{\binom{n}{n/2}} \ \mbox{ for even } \ n.
\end{equation} 
Both equations (\ref{eq2b}) and (\ref{eq2c}) can be solved by applying the method of intensively monotone operators similarly to the 
proof of the previous Theorem \ref{th1}. 
\end{proof}

\end{section}

\section{Characterization by independence property}\label{sec2} 
\setcounter{equation}{0} 

\begin{thm}\label{th2}
Let $X_1, X_2$ be independent identically distributed random variables having a symmetric distribution. Let $\varepsilon$ be a random variable taking two values, $0$ and $1,$ each with probabilities $tfrac12$ and such that $X_1,X_2,\varepsilon$ are independent. Define two forms
\[ L_1 = \varepsilon X_1 +(1-\varepsilon) X_2 \quad \mbox{ and } \quad L_2 = \varepsilon X_1 -(1-\varepsilon)X_2.  \]
Then the forms $L_1$ and $L_2$ are independent if and only if $X_1$ takes only two values, $-a$ and $a,$ for some $a > 0$, each with probability $\tfrac12$. 
\end{thm}
\begin{proof}
Let us calculate the joint characteristic function of  $L_1$ and $L_2$. We have
\[ 
{\bf E}[\exp\{i s L_1 + i t L_2\} = {\bf E}[\exp\{i(s+t) \varepsilon X_1 + i(s-t) (1-\varepsilon) X_2\}]
\]
\[
 = \tfrac12(f(s+t) + f(s-t).
\]
The independence property of $L_1$ and $L_2$ is equivalent to the relation 
\begin{equation}\label{eq3}
\tfrac12(f(s+t) + f(s-t)) = f(s)f(t) \ \mbox{ for all } \ s, t.
\end{equation}

It is easy to see that the function $f(t)=\cos(at)$ is a solution of (\ref{eq3}). To see that there are no other solutions, 
 let us put $t=s$ into (\ref{eq3}). We obtain
\[ \tfrac12(f(2s)+1)= f^2(s), \]
what coincides with (\ref{eq2}). The result follows now from Theorem \ref{th1}.
\end{proof}

\section{Characterization of the uniform distribution}\label{sec3}
\setcounter{equation}{0}

Here we give a characterization of the uniform distribution by the property of identical distributions of linear forms with random coefficients.
 
\begin{thm}\label{th3}
Let $X_1,X_2,X_3$ be independent and identically distributed random variables, 
copies of $X$, where $X$ is a non-degenerate random variable symmetric with respect to the origin $0$ and finite 
absolute third moment. Let further, $\xi$ be a random variable taking only two values, $1/2$ and $1,$ each with probability $\tfrac12.$  The random variables $X_,X_2,X_3,\xi$ are assumed to be independent. The following linear  
\[ L_1=\xi X_1+\tfrac{1}{4}(X_2+X_3) \quad \text{ and } \quad L_2=\tfrac12(X_1+X_2+X_3)\]
are identically distributed if and only if $X$ has a uniform distribution on a symmetric interval $(-A,A)$ for some $A>0$.
\end{thm}
\begin{proof}
The  identity of the distributions of $L_1$ and $L_2$ is expressed in terms of the characteristic functions $f$ of $X$ as follows: 
\begin{equation}\label{eq4}
\frac{1}{2}\Bigl(f(t)+f(t/2)\Bigr) f^2(t/4) = f^3(t/2),
\end{equation}
It is easy to see that the function $g(t)= \sin(At)/(At)$ is a solution of (\ref{eq4}) for any $A>0$. To finish the proof it is sufficient to show that the equation (\ref{eq4}) has no other solutions. 

We introduce the function 
\[ {\mathcal{K}}(t)=\frac{f(t)}{f(t/2)},\]
which is well-defined in a neighborhood of the origin $0$ and that ${\mathcal{K}}(t)$ satisfies the equation
\begin{equation}\label{eq5}
{\mathcal{K}}(t)=2 {\mathcal{K}}^2(t/2)-1 \ \mbox{ for all } \ t. 
\end{equation}
Notice that the equation (\ref{eq5}) coincides with (\ref{eq2}). However, we cannot apply the statement of Theorem \ref{th1} because ${\mathcal{K}}(t)$ is not supposed to be a characteristic function. Therefore, we need to propose another method of solving (\ref{eq5}). 

The properties of the random variable $X$ guarantee that its characteristic function $f(t)$ is real-valued and symmetric 
with respect to $0$, and has at least three continuous derivatives. We easily see  that the function ${\mathcal K}(t)$ possesses the same properties.
Therefore,  ${\mathcal K}(0)=1$ and ${\mathcal K}^{\prime}(0)=0$. It is also obvious that if ${\mathcal K}(t)$ is a solution of 
 (\ref{eq5}), then ${\mathcal K}(at)$ is also a solution of this equation for any $a>0$.
However, ${\mathcal K}^{\prime \prime}(0)= \frac{3}{4}f^{\prime \prime}(0) <0,$ hence  we can choose $a>0$ such that that 
${\mathcal K}^{\prime \prime}(0)=-1$. 

Introduce a set $\go F$ of three times continuously differentiable symmetric functions, defined on $[-T,T]$  for some $T>0$ sufficiently small,  and such that for any $g \in {\go F}$ the following properties hold: $g(0)=1$, $g^{\prime \prime}(0)=-1$ and $|g(t)| \leq 1$ for all $|t| \leq T$. Define
\begin{equation}\label{eq6}
d(g_1,g_2) = \sup_{t}\Bigl|\frac{g_1(t)-g_2(t)}{t^3}\Bigr|, \quad g_1,g_2 \in {\go F}.
\end{equation}
It is easy to check that $d(g_1,g_2)$ is a distance in $\go F$, and moreover, that $\go F$ is a complete metric space with this distance. Let us consider operator $\mathcal{B}$ defined as
\[ ({\mathcal{B}}g)(t) = 2 g^2(t/2)-1, \ \mbox{ for all } t.  \]
We can easily verify that $\mathcal{B}$ maps $\go F$ into $\go F$. For arbitrary $g_1,g_2 \in {\go F}$ we have
\[ d({\mathcal{B}}g_1,{\mathcal{B}}g_2) = \sup_{t}\Bigl| 2\cdot\frac{g_1^2(t/2)-g_2^2(t/2)}{t^3}\Bigr|= \sup_{\tau}\Bigl| \frac{1}{4}\cdot\frac{g_1^2(\tau)-g_2^2(\tau)}{\tau^3}\Bigr| \leq \]
\[\leq \frac{1}{2}\cdot \sup_{\tau} \Bigl|\frac{g_1(\tau)-g_2(\tau)}{\tau^3} \Bigr| = \frac{1}{2}\cdot d(g_1,g_2).\]
These relations  show that the operator $\mathcal{B}$ is a contraction and, therefore, it has a unique fixed point in $\go F$. 
Consequently, $\mathcal{K}(t) = \cos (t)$. 

Let us remind that we have assumed that  $\mathcal{K}^{\prime \prime }(0) =-1$. The general solution is $\cos (a t)$ for any $a>0$.
The function $\mathcal{K}(t)$ was defined as $ \mathcal{K}(t) = f(t)/f(t/2)$, so that
\[f(t) =\cos(at) f(t/2) = \]
\[
 = \cos (at) \cos(at/2) f(t/4) = \ldots = \prod_{k=0}^{\infty} \cos (at/2^k) = \frac{\sin(2at)}{2at},
\]
because $f$ is continuous at zero and $f(0)=1$.
\end{proof}

\section{One more characterization}\label{sec4}
\setcounter{equation}{0}

In  Section \ref{sec3} we have given a characterization of the uniform distribution with the characteristic function 
(in normalized form) $f(t)=\sin (t)/t$. Here we provide a characterization of another distribution whose characteristic function 
is $f(t) = t/\sinh(t)$. The latter corresponds to a random variable $Y$ on $\mathbb R$ with a density function  
\begin{equation}\label{eq8}
p(x)=\frac{\pi}{4\cosh^2(\pi x/2)}, \ x \in {\mathbb R.}
\end{equation}

\begin{thm}\label{th4}
Let $X_1,X_2,X_3$ be independent and identically distributed random variables, 
copies of $X$, where $X$ is a non-degenerate random variable, symmetric with respect to the origin, and
 with finite absolute third moment. Let further, $\xi$ be a random variable taking only two values, 
 $1/2$ and $1,$ each with probability $\tfrac12.$ Suppose also that the random variables $X_,X_2,X_3,\delta$ are independent. 
Then the following linear forms 
\[ L_1 =\xi X_1 + \tfrac12(X_2+X_3) \quad \text{ and } \quad L_2 = X_1 +\tfrac{1}{4}(X_2+X_3) \]
are identically distributed if and only if the density function of $X$ is of the form $p(x/a)/a, \ x \in {\mathbb R}$, $a>0$, where $p(x)$ is 
defined by (\ref{eq8}). 
\end{thm}
The proof of Theorem \ref{th4} is very similar to that of the previous Theorem \ref{th3} and, therefore, is omitted.

\section*{Acknowledgment}

The work by Lev Klebanov was partially supported by Grant GA\v{C}R 19-04412S.

\end{document}